\documentclass[12pt]{amsart}

\usepackage[ansinew]{inputenc}
\usepackage{pst-all,epsfig}
\usepackage{float}
\usepackage{graphicx, graphics}
\usepackage[USenglish]{babel}
\usepackage{amsmath}
\usepackage{amsfonts}
\usepackage{amssymb}
 
\usepackage[left=2.5cm,top=2cm,right=2.5cm,bottom=2.5cm]{geometry}
\newcommand{\R}{\mathbb{R}}

\newcommand{\Rn}{\mathbb{R}^{n}}
\newcommand{\Sn}{\mathbb{S}^{n-1}}
\newcommand{\Sd}{\mathbb{S}^{2}}

\newcommand{\inte}{\operatorname*{int}}

\newcommand{\aff}{\operatorname*{aff}}

\newcommand{\bd}{\operatorname*{bd}}

\newtheorem{problem}{Problem}

\newtheorem{lemma}{Lemma}
\newtheorem{theorem}{Theorem}

\newtheorem{conjecture}{Conjecture}

\title{Characterizations of ellipsoids by means of the strong intersection property}

\author{E. Morales-Amaya}
\thanks{ }
\address{School of Mathematics, Autonomous University of Guerrero, Carlos E. Adame 54, Col. Garita C.P. 39650, Acapulco, Guerrero. Mexico.}

\begin{document}

\maketitle

\begin{abstract}  
Let $E_1,E_2\subset \Rn$  be two homothetic solid ellipsoids, $n\geq 3$,  
with center at the origin $O$ of a system coordinates of $\Rn$, and $E_1\subset \inte E_2$. Then there exists a $O$-symmetric ellipsoid $E_3$ such that $E_3$ is homothetic to $E_1$ and,
for all $x\in \bd E_2$, there exists an hyperplano $\Pi(x)$, $O\in \Pi(x)$, such that the relation
\begin{eqnarray}\label{espada}
S(E_1,x)\cap S(E_1,-x)= \Pi(x) \cap E_3.
\end{eqnarray} 
holds, where $S(E_1,x)$ and $S(E_1,-x)$ are the supporting cones of $E_1$ with apex $x$ and $-x$, respectively.

In this work we prove that aforesaid condition characterizes the ellipsoid. In fact, we prove that if $K,S, G\subset \Rn$ are three convex bodies, $n\geq 3$, $O\in \inte K$, $K\subset \inte G\subset \inte S$ and $G$ strictly convex and, for all $x\in \bd S$, there exists $y\in \bd S$, $O$ in the line defined by $x,y$, an hyperplane $\Pi(x)$, $O\in \Pi(x)$, such that the relation
\begin{eqnarray}\label{fuerza}
S(K,x)\cap S(K,y)= \Pi(x) \cap \bd G.
\end{eqnarray} 
holds, where $S(K,x)$ and $S(K,y)$ are the supporting cones of $K$ with apex $x$ and $y$, respectively, then $G,K$ and $S$ are $O$-symmetric homothetic ellipsoids.

In this case, we say that the convex body $K$ has the \textit{strong intersection property} relative to $O$ and $S$ and with \textit{associated} body $G$. Thus our main result affirm that if the convex body $K$ has the  strong intersection property relative to $O$ and $S$ and with associated strictly convex body $G$, then $K,S$ and $G$ are concentric homothetic ellipsoids. 
\end{abstract}
\section{Introduction.}  
Let $\mathbb{R}^{n}$ be the Euclidean space of dimension $n$ endowed with the usual inner product $\langle \cdot, \cdot\rangle : \mathbb{R}^{n} \times \mathbb{R}^{n} \rightarrow \R$. We take an orthogonal system of coordinates $(x_1,...,x_{n})$ for  $\mathbb{R}^{n}$ and we denote by $O$ its origin. Let $B(n)=\{x\in \mathbb{R}^{n}: ||x||\leq 1\}$ be the $n$-ball of radius $1$ centered on the origin, and let $\mathbb{S}^{n-1}=\{x\in \mathbb{R}^{n}: ||x|| = 1\}$ be its boundary. For $u\in \Sn$ we denote by $u^{\perp}$ the hyperplane orthogonal to $u$. A set $K\subset \Rn$ is said to be a \textit{convex body} if it is compact convex  set with non-empty interior. An excellent reference for the basic concepts and results of convexity is the book \cite{gardner}.  The line and the line segment defined by the point $x,y\in \Rn$ will be denoted by $L(x,y)$ and $[x,y]$, respectively.

A chord $[p,q]$ of a convex body $K$ is called a \textit{diametral chord} of $K$, if there are parallel support hyperplanes of $K$ at $p$ and $q$.

Let $H$ be an hyperplane and let $x,y\in \Rn$, $x\not=y$. Let $R^H_{xy}: \Rn \rightarrow \Rn$ be the affine reflection with respect to $H$ and parallel to the line $L(x,y)$.

Let $K\subset \Rn$ be a convex body. Given a point $x \in \Rn \backslash K$ we denote the cone generated by $K$ with apex $x$ by $C(K, x)$, that is, $C(K, x) := \{x + \lambda(y - x) : y \in K, \lambda \geq  0\}$, by $S(K, x)$ the boundary of $C(K, x)$, in other words, $S(K, x)$ is the support cone of $K$ from the point $x$ and by 
$\Sigma(K, x)$ the graze of $K$ from $x$, that is, $\Sigma(K, x) :=S(K, x)\cap  \bd K$.  

Let $K, S\subset \Rn$ be a convex bodies, $n\geq 3$, $K\subset \inte S$. Suppose that, for every $x\in \bd S$,  the set $\Sigma(K, x)$ is contained in a hyperplane. It has been conjectured that such condition implies that the convex body is an ellipsoid. In \cite{gjmc2} was proved such conjecture with additional conditions: $K$ and $S$ are $O$-symmetric and $\bd S$ \textit{is far enough to} $\bd K$. In that work was observed that, for every $x\in \bd S$, the set $S(K,x)\cap S(K,-x)$ is contained in an hyperplane (See Lemma 2 of \cite{gjmc2}).  This observation motives the following 
definition: We say that the convex body $K\subset \Rn$, $n\geq 3$ has the \textit{intersection property} in dimension $n$ if there exists a point 
$O\in \inte K$ and a convex body $S\subset \Rn$, $K\subset S$, and, for every $x\in \bd S$, there exists $y\in \bd S$, $O\in L(x,y)$ an hyperplane $\Pi(x)$, $O\in \Pi(x)$, with the property that the relation
\[
S(K,x)\cap S(K,y)\subset \Pi(x),
\] 
holds. 

We say that the convex body $K\subset \Rn$, $n\geq 3$ has the \textit{strong intersection property} in dimension $n$ if it has the intersection property for $O\in \inte K$ and the convex body $S\subset \Rn$, $K\subset S$, and, furthermore, there exists a convex body $G$ such that $K\subset G$ and, for every 
$x, y\in \bd S$, $O\in L(x,y)$, the relation
\begin{eqnarray}\label{emma}
S(K,x)\cap S(K,y)= \Pi(x) \cap G,
\end{eqnarray}
holds. In this case, we say that the convex body $K$ has the \textit{strong intersection property} relative to $O$ and $S$ and with \textit{associated} body $G$.

In this work we will star stating two results which represent a property of the ellipsoid in terms of the intersections of pairs of support cones of a convex body (Theorems \ref{tokio} and \ref{emperador}). By completeness, we will give the proofs of this two results. Furthermore, our main result is Theorem \ref{valor} which affirm that if the convex body $K\subset \Rn$, $n\geq 3$, has the strong intersection property relative to $O\in \inte K$ and the convex body $S$, $K\subset \inte S$, and with associated strictly convex body $G$, $K\subset \inte G\subset \inte S$, then $K$, $S$ and $G$ are concentric homothetic ellipsoids. 

In order to prove theorem \ref{valor}, on the one hand, we first show, in Theorem \ref{great}, that if the convex body $K\subset \Rn$, $n\geq 3$, has the strong intersection property relative to $O\in \inte K$ and the convex body $S$, $K\subset \inte S$, and with associated strictly convex body $G$, $K\subset \inte G \subset \inte S$ is because $K,S$ and $G$ are $O$-symmetric (notice that we require assume that $G$ is strictly convex), this is carried out by a series of lemmas and, finally, it is shown that K is centrally symmetric (for which a characterization of central symmetry demonstrated in \cite{efrenconos} is used) and, on the other hand, we use the Theorem \ref{chiquita}, where additionally to the strong intersection property, is  assumed that if some of the convex bodies $K,S$ and $G$ is an ellipsoid, then the other two are ellipsoids too.  
 \section{Statement of the results.} 
We will start presenting two results relatives to ellipsoids, the Theorems \ref{tokio} and \ref{emperador}. In order to do this we need the following definitions.

Let $S\subset \Rn$ be an embedding of $\Sn$ in $\Rn$. By the Jordan's Curve Theorem in $n$ dimension (reference), $S$ divides $\Rn$ in two components, we will call the bounded component as the \textit{interior} of $S$ and it will be denoted by $\inte S$. 

Given $x\in \Rn$, we denote by 
$\overrightarrow{Ox}$ the ray defined by $x$, i.e., 
$\overrightarrow{Ox}=\{\lambda x :  \lambda \geq  0\}$. The set $S$ is said to be a $O$-\textit{star} if in every ray, starting in $O$, there exists a point of $S$ and such point is unique. Let $S\subset \Rn$ be a $O$-star set. We consider a map $\phi: S \rightarrow S$  such that, for $x\in S$, $\phi(x)$ is defined as the point in $S$ such that $\overrightarrow {O\phi(x)}$ has the opposite direction of the ray $\overrightarrow{Ox}$. Notice that if $S$ is $O$-symmetric, then $\phi(x)=-x$.

\begin{theorem} \label{tokio} 
Let $E\subset \Rn$  be an $O$-symmetric ellipsoid, $n\geq 3$,   and let $S\subset \Rn$ be an embedding of $\Sn$ in $\Rn$ such that $S$ is $O$-star and $E\subset \inte S$. Then for all $x\in S$ there exists an hyperplane $\Pi(x)$, $O\in \Pi(x)$, such that the relation
\begin{eqnarray}\label{japon}
S(E,x)\cap S(E,\phi(x))\subset \Pi(x).
\end{eqnarray} 
holds.
\end{theorem}
An interesting particular case of the Theorem \ref{tokio} is the following result which was mentioned in the abstract. 
\begin{theorem} \label{emperador} 
Let $E_1,E_2\subset \Rn$  be two $O$-symmetric homothetic ellipsoids, $n\geq 3$,  and $E_1\subset \inte E_2$. Then there exists a $O$-symmetric ellipsoid $E_3$ such that $E_3$ is homothetic to $E_1$ and,
for all $x\in E_2$, there exists an hyperplano $\Pi(x)$, $O\in \Pi(x)$, such that the relation
\begin{eqnarray}\label{espada}
S(E_1,x)\cap S(E_1,-x)= \Pi(x) \cap E_3.
\end{eqnarray} 
holds. Furthermore, let $E_2=\lambda E_1, \lambda> 0$. If $\lambda=\sqrt{2}$, then $E_2=E_3$, if $\sqrt{2}<\lambda$, then $E_3\subset E_2$ and if $\lambda<\sqrt{2}$, then $E_2\subset E_3$.
\end{theorem}
The following problems arise of natural manner.
\begin{conjecture}\label{espada}
Let $K\subset \Rn$ be a convex body, $n\geq 3$,and let $S\subset \Rn$ be an embedding of $\Sn$ in $\Rn$ such that $S$ is $O$-star, $O\in \inte K$ and $K\subset \inte S$. Then for all $x\in S$ there exists an hyperplane $\Pi(x)$, $O\in \Pi(x)$, such that the relation
\begin{eqnarray}\label{armadura}
S(K,x)\cap S(K,\phi(x))\subset \Pi(x).
\end{eqnarray} 
holds.
Then $K$ is an ellipsoid.
\end{conjecture} 
\begin{problem}
To prove or disproof Conjecture \ref{espada} assuming that $K$ and $S$ are $O$-symmetric.
\end{problem}
\begin{theorem} \label{valor} 
Let $K,S, G\subset \Rn$  be three convex bodies, $n\geq 3$, $O\in \inte K$ and $K\subset \inte G\subset \inte S$. Suppose that $K$ has the strong intersection property relative to $O$ and $S$ and with associated strictly convex body $G$. Then $K, S$ and $G$ are $O$-symmetric homothetic ellipsoids.
\end{theorem}
In \cite{egj} was proved the rather special case of the Theorem \ref{valor} when $G=S$.

\begin{theorem} \label{great} 
Let $K,S, G\subset \Rn$  be three convex bodies, $n\geq 3$, $O\in \inte K$  and $K\subset \inte G\subset \inte S$. Suppose that $K$ has the strong intersection property relative to $O$ and $S$ and with associated strictly convex body $G$. Then $G,K$ and $S$ are $O$-symmetric.
\end{theorem}

\begin{theorem} \label{chiquita} 
Let $K,S, G\subset \Rn$  be three convex bodies, $n\geq 3$, $O\in \inte K$  and $K\subset \inte G\subset \inte S$. Suppose that $K$ has the strong intersection property relative to $O$ and $S$ and with associated strictly convex body $G$. Furthermore, suppose that some of the bodies $K, S$ and $G$ is an ellipsoid. Then the other two bodies are ellipsoids and  $K, S$ and $G$ are homothetic.
\end{theorem}

\section{Proof of Theorems \ref{tokio} and \ref{emperador}.}  
Let $G_1,G_2 \subset \Rn$  be two homothetic ellipsoids $O$-symmetric $G_2\subset G_1$, $n\geq 3$, let $x\in \mathbb{R}^{n+1}$ and let $y\in  L(O,x)$, $x\not=y$. We denote by $C_x(G_1)$, $C_y(G_2)$ the cones defined by $G_1$ and $x$ and $G_2$ and $y$, respectively, that is, $C_x(G_1) := \{x + \lambda(z - x) : z \in G_1, \lambda \geq  0\}$,  
$C_y(G_2) := \{y + \lambda(z - y) : z \in G_2, \lambda \geq  0\}$. In order to prove the Theorem \ref{tokio} we need the following lemma. 
\begin{lemma}\label{alegria}
The intersection $C_x(G_1)\cap C_y(G_2)$ is contained in ah hyperplane. 
\end{lemma}
\begin{proof}
For all $\lambda \in \R$, the sections $\Pi_{\lambda}\cap C_x(G_1)$ and $\Pi_{\lambda}\cap C_y(G_2)$ are homothetic ellipsoid with centres at $L(O,x)$, where $\Pi_{\lambda}:=\{(x_1,...,x_{n+1})\in \mathbb{R}^{n+1}: x_{n+1}=\lambda\}$. Let $\lambda_0$ be a real number such that 
$\Pi_{\lambda_0}\cap C_x(G_1) \cap C_y(G_2)\not=\emptyset$. Then the homothetic sections $\Pi_{\lambda_0}\cap C_x(G_1)$, $\Pi_{\lambda_0}\cap C_y(G_2)$ are concentric 
 and it have a common point. Thus 
 \[
 \Pi_{\lambda_0}\cap C_x(G_1)=\Pi_{\lambda_0}\cap C_y(G_2).
 \]
From here, it is clear that $C_x(G_1)\cap C_y(G_2)\subset \Pi_{\lambda_0}$.
\end{proof}
\textbf{Proof of Theorem \ref{tokio}.} For $x\in \Rn$ we denote by $\Gamma_x$ the polar hyperplane of $E$ corresponding to the pole $x$. Notice that
\[
\Sigma(E,x)=\Gamma_x\cap E \textrm{  }\textrm{ and } \textrm{  }\Sigma(E,\phi(x))=\Gamma_{\phi(x)}\cap E.
\]
Furthermore, since $\phi(x)\in L(O,x)$, the hyperplanes $\Gamma_x$ and 
$\Gamma_{\phi(x)}$ are parallel (referencia). The Theorem \ref{tokio} will follow from Lemma \ref{alegria} applied to the homothetic and concentric ellipsoids  $\Gamma_x \cap E$ and $S(E,\phi(x))\cap \Gamma_x$ which defined the cones $S(E,x)=C_x(\Sigma(E,x))$ and $S(E,\phi(x))=C_{\phi(x)}(\Sigma(E,\phi(x)))$.
\qed

\textbf{Proof of Theorem \ref{emperador}}. Let $A:\Rn \rightarrow \Rn$ be an affine map such that $A(E_2)=\Sn$ and $\bar{E}_1:=A(E_1)$ is a sphere concentric with $\Sn$. By virtue of the symmetry of the sphere, it follows that, for every $x\in \mathbb{S}^n$, the  set
$S_x:=S(\bar{E}_1,x)\cap S(\bar{E}_1,-x)$ is a sphere in $x^\perp$. It is clear the $\bar{E}_3=\bigcup_{x\in \Sn}S_x$ is a sphere concentric with $\Sn$ (notice that, for every $x\in \mathbb{S}^n$, the relation $S_x=\bar{E}_3 \cap x^\perp $ holds). Thus $E_3:=A^{-1}(\bar{E}_3)$ is the ellipsoid which satisfies the condition of Theorem \ref{emperador}. 

On the other hand, by virtue that $E_2=\lambda E_1$ it follows that $\Sn=\lambda \bar{E}_1$.  If $\lambda =\sqrt{2}$, then $S_x=\mathbb{S}^n \cap x^{\perp}$ (Notice that, in dimension 2, $\bar{E}_1$ is inscribed in the square inscribed in $\mathbb{S}^1$ and $S_x$ is the diameter perpendicular to $x$). Thus 
$\bar{E}_3=\mathbb{S}^n$ and, consequently, $E_3=E_2$. If $\sqrt{2}<\lambda$, then, for every $x\in \mathbb{S}^n$, $ S_x\subset  \mathbb{S}^n \cap x^{\perp}$. Therefore $   \bar{E}_3\subset \mathbb{S}^n$, i.e., $E_3\subset E_2$. If $ \lambda<\sqrt{2}$, then, for every $x\in \mathbb{S}^n$, $\mathbb{S}^n \cap x^{\perp}\subset S_x$. Hence $\mathbb{S}^n\subset \bar{E}_3$, i.e., $E_2\subset E_3$.
\qed
 
\section{Proof of Theorem \ref{great} for dimension 3.} 
In the proof of the Theorems \ref{great} we will assume that $O$ is the origin of a system of coordinates. The proof that $K$ is centrally symmetric for the case $n=3$ will be given in a serie of steps: 
\begin{itemize}
\item [i)] We will prove, in the Lemma \ref{porton}, that if
the convex body $K$ has the strong intersection property relative to the point $O\in \inte K$ and the body $S$, $K\subset \inte S$, and with associated strictly convex body $G$, $K\subset \inte G\subset \inte S$, then the body $S$ is centrally symmetric. 

\item [ii)] In Lemma \ref{ensueno} we demonstrate that the body $S$ is strictly convex.
 
\item [iii)] In the Lemma \ref{lanza} we will prove, that if $x,y\in \bd S$, for which $O\in L(x,y)$, and there exists an affine reflexion, with respect to the hyperplane $H$ and parallel to $L(x,y)$, such that it maps the cone $C(K,x)$ in to the cone $C(K,y)$, this affine reflexion sent the graze 
$\Sigma(K,x)$ in to the graze $\Sigma(K,y)$. 
\item [iv)] In Lemma \ref{muneca}, we will prove a kind of \textit{symmetry} with respect to plane of affine reflexion mentioned in the Lemma \ref{lanza}, i.e.,

\textit{Let $p,q\in \bd S$ such that $O\in L(p,q)$ and there exists a plane $\Lambda$, $O\in \Lambda$, and an affine reflexion $R^{\Lambda}_{pq}: \Rn \rightarrow \Rn$ for which 
\[
R^{\Lambda}_{pq}(C(K,p))=C(K,q).
\]
If, for $x,y\in \bd S$, $O\in L(x,y)$ and $L(x,y)\subset \Lambda$, there exists a plane $H$ and an affine reflexion $R^H_{xy}: \Rn \rightarrow \Rn$ such that 
\[
R^H_{xy}(C(K,x))=C(K,y),
\]
then the line $L(p,q)$ is contained in $H$.
}

\item [v)] The convex bodies $K$ has the strong intersection property relative the point $O$ and $G$ with associated body $S$. 
\end{itemize}

The next theorem is due to Hammer \cite{Hammer} and it will be used in the proof of the Lemma \ref{porton}.

\textit{Let $K\subset \Rn$, $n\geq 2$, be a convex body. If every chord through $O \in K$ is a diametral chord, then K is centrally symmetric with center at $O$.}

\begin{lemma}\label{porton}
Let $K,S\subset \Rn$  be two convex bodies, $n\geq 3$, $O\in \inte K$ and $K\subset \inte S$. Suppose that the convex body $K$ has the strong intersection property relative to the point $O$ and the body $S$ and with associated strictly convex body $G$, $K\subset \inte G$. Then, the body $S$ is centrally symmetric.
\end{lemma}
\begin{proof}
In order to prove that the body $S$ is centrally symmetric we are going to prove that every chord $[a,b]$ of $S$ with $O\in [a,b]$ is a diametral chord. In this case, by the Theorem of Hammer $S$ is centrally symmetric. 

We introduce some notation. For every $x \in \bd S$, we denote by $\Pi_x$ the plane such that $\Pi_x\cap G=S(K,x)\cap S(K,y)$ given by the definition of strong intersection property, by $G_x$ the section $\Pi_x\cap G$ and by $\Gamma_x$ the plane through $x$ parallel to $\Pi_x$. Notice that $\Pi_x=\Pi_y$ and, consequently, $\Gamma_x$ and $\Gamma_y$ are parallel. On the other hand, we observe, by virtue that we can interprete $G_z$ as the projection of $K$ from $z$ onto $\Pi_z$, that  
\begin{eqnarray}\label{grandeza}
K\cap \Gamma_z=\emptyset
\end{eqnarray}
(The body $K$ is inscribed in the cone $S(K,z)$ which has vertex at $z$).

Let $x \in \bd S$, we are going to demonstrate that $\Gamma_x$ is a supporting plane of $S$. Let $L\subset \Gamma_x$ be a line passing through $x$. We will show that $L$ is supporting line of $S$. On contrary, let us assume that there exists a point $x_0\in \bd S$ in $L$, $x_0\not=x$. Let $H\subset \Pi_x$ be a supporting line of $G_x$ parallel to $L$ and which intersect $G_x$ at $w$. Since $l(x,w)$ is supporting line of $K$, there exists $a\in \bd K$ in $l(x,w)$ and the plane $\aff\{x,H\}$ is supporting plane of $K$. 

First, we suppose that $\Pi_x=\Pi_{x_0}$, i.e., $G_x=G_{x_0}$. By virtue that $x\not=x_0$, it follows that $l(x,a)\not=l(x_0,a)$. Thus the point $\bar{w}:=l(x_0,a)\cap H$ is such that $\bar{w}\in  G_{x_0}$ and $\bar{w}\not=w$. Since $H$ is supporting line of $G_x$ it follows that $[w,\bar{w}]\subset G_x$ but this contradicts the strictly convexity of $G$.
  
Now we suppose that $\Pi_x\not=\Pi_{x_0}$. Since the plane $\aff\{x,H\}$ is supporting plane of $K$, the line $\bar{H}:=\aff\{x,H\} \cap \Pi_{x_0}$ is supporting line of $G_{x_0}$ and it is passing through $w$. By virtue that $x\not=x_0$, it follows that $l(x,a)\not=l(x_0,a)$. Thus the point $\bar{w}:=l(x_0,a)\cap \bar{H}$ is such that 
$\bar{w}\in  G_{x_0}$ and $\bar{w}\not=w$. Given that $\bar{H}$ is supporting line of $G_{x_0}$ and $w,\bar{w}\in \bar{H} \cap G_{x_0}$, it follows that $[w,\bar{w}]\subset  G_{x_0}$ but contradicts the strictly convexity of $G$. 

This completes the proof the $\Gamma_z$ is a supporting plane of $S$. 
\end{proof} 

\begin{lemma}\label{ensueno}
The body $S$ is strictly convex.
\end{lemma}
\begin{proof} 
On the contrary to the Lemma statement, let us assume that $S$ is not strictly convex, that is, we assume that there exists a line segment $[a,b]\subset \bd S$, $a\not=b$. Let $z\in \inte[a,b]$. By Lemma \ref{porton}, $\Pi_z$ and $\Gamma_z$ are parallel and $[a,b]\subset \Gamma_z$, otherwise, $a$ and $b$ would be in different half spaces of the two defined by $\Gamma_z$. Now we procede in analogous way as in the proof of Lemma \ref{porton} and rich to the contradiction. Then $S$ is strictly convex.
\end{proof} 

The Lemmas  \ref{lanza} and \ref{muneca} below, used in the proof of Lemma \ref{silvestre}, are in the spirit of the next result \cite{efrenconos}, which will be used in the proof of Theorem \ref{great} (from our point of view, it is interesting and convenient to present it in terms of affine reflexions).

\textbf{Characterization of central symmetry.}

\textit{Let $K\subset \Rn$, $n\geq 3$ be a strictly convex body and let $S\subset \Rn$ be a hypersurface which is the image of an embedding of the sphere $\mathbb{S}^{n-1}$, such that $K$ is contained in the interior of $S$. Suppose that, for every $x \in S$, there exists $y\in S$ such that 
the support cones $S(K,x)$ and $S(K,y)$ differ by a central symmetry. Then $K$ and $S$ are centrally symmetric and concentric.}

\begin{lemma}\label{lanza}
Let $S$, $K$ be two convex bodies in $\Rn$, $n\geq 3$, $K$ strictly convex and let $O\in \inte K$. Suppose that $K\subset \inte S$ and for every pair of points $p,q\in \bd S$, for which $O\in L(p,q)$, there exists a plane $\Lambda$ and an affine reflexion $R^{\Lambda}_{pq}: \Rn \rightarrow \Rn$ such that
\begin{eqnarray}\label{sol}
R^{\Lambda}_{pq}(C(K,p))=C(K,q). 
\end{eqnarray} 
Then 
\begin{eqnarray}\label{angeles}
R^{\Lambda}_{pq}(\Sigma(K,p))=\Sigma(K,q),
\end{eqnarray}
\end{lemma}
\begin{figure}[H]
    \centering
    \includegraphics[width=.8\textwidth]{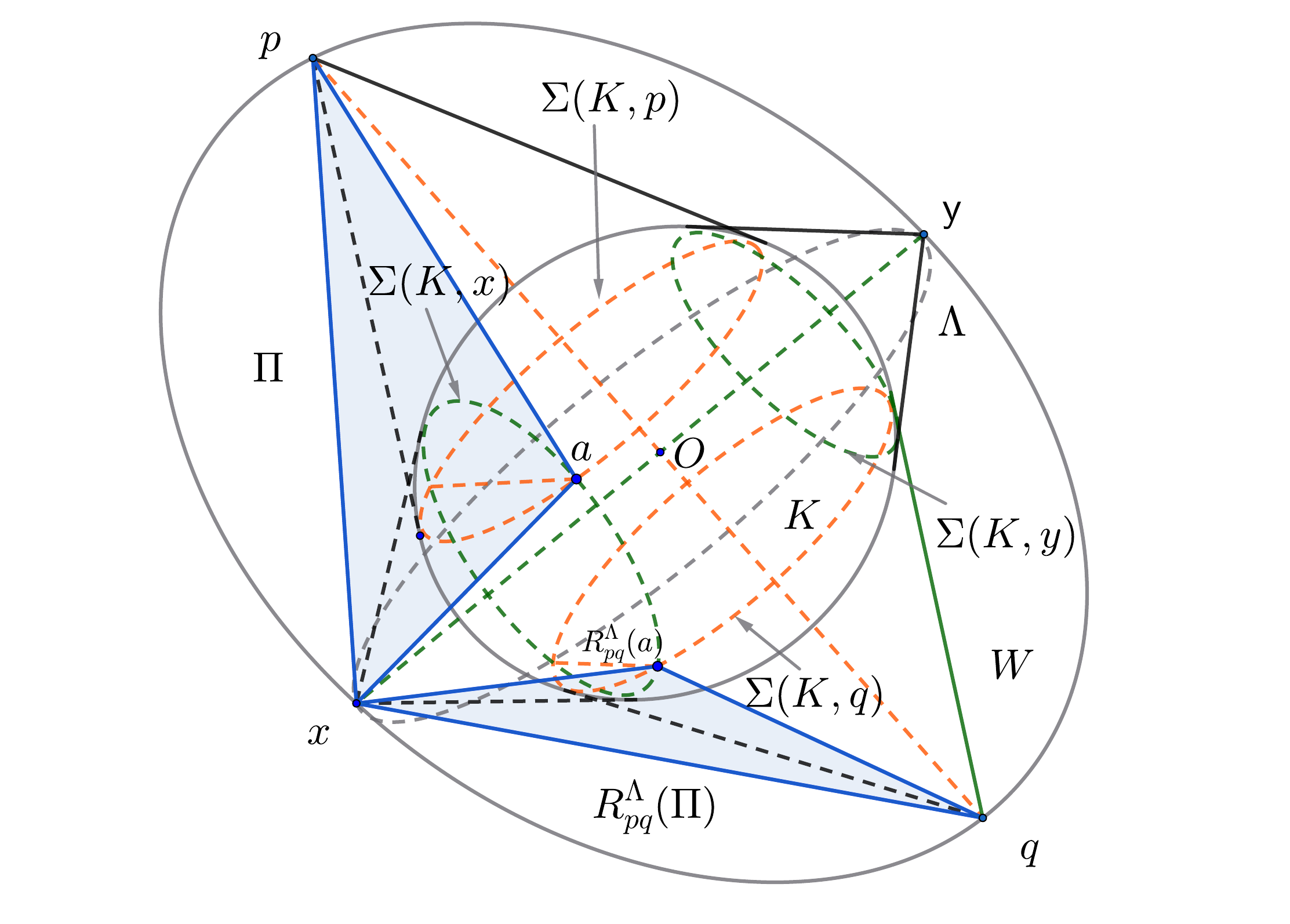}
    \caption{The relation $R^{\Lambda}_{pq}(\Sigma(K,p))=\Sigma(K,q)$ holds.}
    \label{fanta}
\end{figure}
\begin{proof}
From the relation (\ref{sol}) is ease to see that $S$ has center at $O$.
Let $x,y\in \Lambda \cap \bd W$ with $O\in L(x,y)$, let $\Pi$ be a support plane of $K$ containing the line $L(p,x)$ and let $a\in \bd K \cap \Pi$. Notice that $\Pi$ is support plane of $C(K,p)$ and $C(K,x)$. From 
(\ref{sol}) it follows that $R^{\Lambda}_{pq}(\Pi)$ is support plane of $C(K,q)$. On the other hand, since $x\in \Lambda \cap \bd S$, the plane $R^{\Lambda}_{pq}(\Pi)$ is support plane of $C(K,x)$ (see Fig. \ref{fanta}). Thus $R^{\Lambda}_{pq}(a)\in \Sigma(K,x))\cap \Sigma(K,q)$, i.e., $R^{\Lambda}_{pq}(a)\in \Sigma(K,q)$. 
\end{proof}
With the notation above we present the following lemma.
\begin{lemma}\label{muneca}
Let $x,y\in \Lambda \cap \bd S$ with $O\in L(x,y)$ and let $R^H_{xy}: \Rn \rightarrow \Rn$ be the affine reflexion, with respect to the hyperplane $H$, $O\in H$, and parallel to $L(x,y)$, such that 
\begin{eqnarray}\label{tamarindo}
R^{H}_{xy}(C(K,x))=C(K,y),
\end{eqnarray}
holds. Then the line $L(p,q)$ is contained in $H$.
\end{lemma}
\begin{proof}
Let $a\in \Sigma(K,p)\cap \Sigma(K,x)$. Notice that by Lemma \ref{lanza} 
\[
R^{\Lambda}_{pq}(a)\in \Sigma(K,q),  R^{H}_{xy}(a)\in \Sigma(K,p)\cap \Sigma(K,y), R^{\Lambda}_{pq}(R^{H}_{xy}(a))\in \Sigma(K,q)
\]
and the lines $L(a, R^{\Lambda}_{pq}(a))$, 
$L(R^{H}_{xy}(a), R^{\Lambda}_{pq}(R^{H}_{xy}(a)))$ are parallel to 
$L(p,q)$ (see Fig. \ref{fanta}). Thus, if we denote by $D_x$, $D_y$ the planes defined by $x,a, R^{\Lambda}_{pq}(a)$  and $y, R^{H}_{xy}(a), R^{\Lambda}_{pq}(R^{H}_{xy}(a))$, respectively, it follows that 
$D_x \cap D_y$ is parallel to $L(p,q)$.

On the other hand, we observe that
\[
L(x,a)\cap L(y,R^{H}_{xy}(a))\in D_x \cap D_y \textrm{   }\textrm{   } \textrm{and} \textrm{  }\textrm{   }
L(x, R^{\Lambda}_{pq}(a))\cap L(y,R^{\Lambda}_{pq}(R^{H}_{xy}(a)))\in D_x \cap D_y
\]
and 
\[
L(x,a)\cap L(y,R^{H}_{xy}(a))\in G_{xy}
\textrm{   }\textrm{   }
\textrm{and} 
\textrm{   }\textrm{   }
L(x, R^{\Lambda}_{pq}(a))\cap R^{\Lambda}_{pq}(R^{H}_{xy}(a))\in G_{xy},
\]
where $G_{xy}:=C(K,x)\cap C(K,y)=H\cap C(K,x)=H\cap C(K,y)$. Hence we conclude that the plane $H$ is the plane defined by $O$ and $D_x \cap D_y$. Consequently $L(p,q)\subset H$.
 \end{proof}
\begin{lemma}\label{silvestre} 
The convex bodies $K$ has the strong intersection property 
relative to the point $O$ and the convex body $G$ and with associated body $S$. 
\end{lemma}
\begin{proof}
Let $u\in \bd G$, we are going to prove that there exists a a point $v\in G$ and a plane $W$, $O\in W$, such that
\begin{eqnarray}\label{olorosa}
C(K,u)\cap C(K,v)=W\cap S.
\end{eqnarray} 
By Lemma \ref{porton}, $S$ is centrally symmetric. Since for every $x\in S$, there exists  a plane $H$, $O\in H $ such that 
\[
C(K,x)\cap C(K,-x)=H\cap G
\]
we can interprete this as there exists an affine reflexion $R^H_{x(-x)}: \Rn \rightarrow \Rn$ with respect to the hyperplane $H$, $O\in H$, and a direction parallel to $L(x,-x)$, such that 
\[
R^{H}_{x(-x)}(C(K,x))=C(K,-x).
\]
Thus we are in conditions to apply Lemmas \ref{lanza} and \ref{muneca}.

Let $v:= L(u,O) \cap \bd G$, $v\not=u$ and let $\{x, -x\} :=L(u,v) \cap \bd S$. Let $\Gamma$ be a plane containing the line $L(u,v)$. We denote by $R_1, R_2\subset \Gamma$ the rays emanating from $u$ which are contained in the supporting lines $L_1$, $L_2$ of $\Gamma \cap K$ passing through $u$ and let $p:=R_1\cap \bd S$ and $q:=R_2\cap \bd S$ (notice that here we use the condition $G\subset \inte S$). By virtue of the hypothesis, there exists a plane $\Lambda$ such that the relation
\[
C(K,p)\cap C(K,q)=\Lambda \cap G
\] 
holds. By our choice of $u$ and $v$ and since $O\in \Lambda$ it is clear that $L(u,v)\subset \Lambda$. By Lemma \ref{muneca}, $L(p,q)\subset H$, where $H$ is the plane such that
\[
C(K,x)\cap C(K,-x)=H \cap G.
\] 
Varying $\Gamma$, assuming that $L(u,v)\subset \Gamma$, we obtain that relation (\ref{olorosa}) holds if we define $W=H$, i.e., $H$ is the plane that we were looking for.
\end{proof}

\textbf{Proof of Theorem \ref{great}}. By Lemma \ref{silvestre}, the convex bodies $K$ has the strong intersection property relative to the point $O$ and the convex body $G$ and with associated body $S$. On the other hand, by the Lemma \ref{ensueno}, the body $S$ is strictly convex. Thus, by the Lemma \ref{porton} applied to the bodies $K$ and $G$, $G$ is centrally symmetric. Then, by virtue that, for $x\in S$, there exists a plane $H$, $O\in H$, such that $C(K,x)\cap C(K,-x)=H \cap G$, the cones 
$C(K,x)$ and  $C(K,-x)$ differ by a central symmetry. Thus, by the characterization of central symmetry of \cite{efrenconos}, the body $K$ is centrally symmetric. Hence the bodies $K$, $S$ and $G$ are centrally symmetric and concentric.

\section{Proof of Theorem \ref{chiquita} for dimension 3.} 
In the proof of the Theorem \ref{chiquita} we will assume that $O$ is the origin of a system of coordinates. The proof is organized as following:
\begin{itemize}
\item [a)] 1. We assume that $K$ is an ellipsoid with center at $O$ and we prove that $G$ is an ellipsoid concentric with $K$, 2. We prove that $K$ and $G$ are homothetic, 3. We prove that $S$ is an ellipsoid with center at $O$ and homothetic to $K$.

\item [b)] 1. We suppose that $G$ is an ellipsoid with center at $O$ y we prove that $K$ is an ellipsoid concentric with $G$, 2. Using 2 and 3 from a) we conclude $K$, $S$ and $G$ are ellipsoids homothetic and concentric. 

\item [c)] 1. We suppose that $S$ is an ellipsoid and we prove that $K$ is an ellipsoid, 2. Using 1 and 2 from a) we conclude that $K$, $S$ and $G$ are ellipsoids homothetic and concentric. 
 
\end{itemize}
Let $C\subset \Rn $ be a convex cone, the cone is said to be \textit{ellipsoidal} if there is a hyperplane $\Pi$ such that $\Pi \cap C$ is an ellipsoid. In the proof of Theorem \ref{chiquita} we will need the following result which was proven in 
\cite{egj} (which can be seen as a particular case of Theorem 2 of \cite{bigru}).
\begin{itemize} 
\item [[MMJ]] \textit{Let $K,G\subset \Rn$ be convex bodies, $n\geq 3$. Suppose that $K\subset \inte G$, $K$ is $O$-symmetric and, for every $x\in \bd G$, the cone $C(K,x)$ is ellipsoidal. Then $K$ is an ellipsoid.}
\end{itemize} 
 
We recall that the we denote, for every $z \in \bd M$, by $\Pi_z$ the plane such that $\Pi_z\cap G=S(K,z)\cap S(K,-z)$, by $G_z$ the section $\Pi_z\cap G$ and by $\Gamma_z$ the plane trough $z$ parallel to $\Pi_z$.  

\textbf{a) 1.} We suppose that $K$ is an ellipsoid. In order to prove that $G$ is an ellipsoid, we are going to prove that all the sections of $G$ passing through $O$ are ellipses. Thus, by Theorem 16.12 in \cite{bu}, it will be deducted that $G$ is an ellipsoid. Let $\Pi$ be a plane through $O$. By a continuity argument, it follows that there exists $z\in \bd S$ such that $\Pi=\Pi_z$. On the other hand, by Lemma \ref{porton}, $S$ is $O$-symmetric. Thus 
$-z$ belongs to $S$. Since $K$ is an ellipsoid the cones  $S(K,z)$, $S(K,-z)$ are ellipsoidal. By the relation $G_z=S(K,z)\cap S(K,-z)$ given by the strong intersection property, it follows that $G_z$ is an ellipse. Thus $G$ is an ellipsoid.

\textbf{a) 2.} Now we are going to demonstrate that $K$ and $G$ are homothetics. In order to do this we will prove that for every plane $\Pi$, $O\in \Pi$, the sections $\Pi \cap K$ and $\Pi\cap G$ are homothetic. Let $\Pi$ be a plane, $O\in \Pi$. Let $z\in \bd S$ such that $\Pi=\Pi_z$. The section $\Delta_z \cap K$ is an ellipse with center at the line $L(O,z)$ and the section $G_z$ has center at $O$. Since $\Delta_z \cap K$ and $G_z$ are sections of the cone $S(K,z)$ it follows that $\Pi_z$ and $\Delta_z$ are parallel. Thus 
$\Delta_z \cap K$ and $G_z$ are homothetic. On the other hand, by virtue that all the parallel section of $K$ are homothetic,  the section $\Delta_z \cap K$ and $\Pi \cap K$ are homothetic. Hence $\Pi \cap K$ and $G_z:=\Pi_z\cap G=\Pi \cap G$ are homothetic.

By a theorem of A. Rogers proved in \cite{Rogers}, $K$ and $G$ are homothetic.

\textbf{a) 3.} Let $A:\mathbb{R}^3 \rightarrow \mathbb{R}^3$ be an affine transformation such that $A(K)$ and $A(G)$ are two concentric spheres. Hence $A(S)$ is the locus of the vertices of right circular cones, where $A(K)$ is inscribed, which are congruent. Consequently $A(S)$ is a sphere with center at $A(O)$.  
 
\textbf{b) 1.}  We assume that $G$ is an ellipsoid.  In order to demonstrate that $K$ is an ellipsoid we are going to prove that, for each $x\in \bd S$, the cone $S(K,x)$ is ellipsoidal and, then, we will apply Theorem [MMJ] to conclude that $K$  is an ellipsoid (Notice that, by Theorem \ref{great}, $K$ is $O$-symmetric). Let $x\in \bd S$. By Lemma \ref{porton}, $S$ is $O$-symmetric. Thus 
$-x$ belongs to $S$. By hypothesis there exists a plane $\Pi_x$, $O\in \Pi_x$, such that the intersection $S(K,x)\cap S(K,-x)$ is equal to $\Pi_x\cap G$. By virtue that $G$ is an ellipsoid, the section $\Pi_x\cap G$ is an ellipse. Thus $S(K,x)$ is ellipsoidal.

\textbf{c) 1}. We assume that $S$ is an ellipsoid. In order to demonstrate that $K$ is an ellipsoid we are going to prove that, for each $x\in \bd G$, the cone $S(K,x)$ is ellipsoidal and, then, we will apply Theorem [MMJ] to conclude that $K$  is an ellipsoid (Notice that, by Theorem \ref{great}, $K$ is $O$-symmetric). The next lemma will be used in the proof that, for $x\in \bd G$, the cone $S(K,x)$ is ellipsoidal.

For $u\in \mathbb{S}^{2}$, we consider the line $L(u):=\{\lambda u: \lambda \in \mathbb{R}\}$ and the set $\Omega_u:=\{z\in \bd S:L(u)\subset \Pi_z\}$.

\begin{lemma}\label{kiss}
For $u\in \Sd$, the relation
\begin{eqnarray}\label{santana}
\Omega_u=S\partial (S,u) 
\end{eqnarray} 
holds
\end{lemma} 
\begin{proof}
Let $u\in \Sd$. Let $z\in \Omega_u$. By Lemma \ref{porton}, the plane $\Gamma_z$ is a supporting plane of $M$ and it is parallel to the line $L(u)$. Thus $z\in S\partial (S,u)$. Hence 
$\Omega_u\subset S\partial (S,u)$. Now let $z\in S\partial (S,u)$. Then there exists a plane $\Gamma$ such that $z\in \Gamma$ and $\Gamma$ is parallel to $u$. Let $\Pi$ be a plane parallel to $\Gamma$ and passing through $O$. Let $\bar{z}\in \Omega_u$ such that $\Pi_z=\Pi$ and $\bar{z}$ is in the same half-space determined by $\Pi$ where is $z$. By Lemma \ref{porton}, the plane $\Gamma_{\bar{z}}$ is a supporting plane of $M$ and it is parallel to $L(u)$. Thus $\Gamma=\Gamma_{\bar{z}}$. By virtue of the strictly convexity of $S$, Lemma \ref{ensueno}, it follows that 
$z=\bar{z}$. Hence $z\in \Omega_u$, i.e., $S\partial (S,u) \subset \Omega_u$. Therefore
$\Omega_u=S\partial (S,u)$.  
\end{proof}
Now we are going to prove that, for $x\in \bd G$, the cone $S(K,x)$ is ellipsoidal. Let $x\in \bd G$. Let $u\in \Sd$ and $L(u) $ be  such that $x\in L(u)$. We claim that, for $y\in S\partial(S,u)$, the line $L(x,y)$ is supporting line of $K$. If $y\in S\partial(S,u)$, by the definition of $\Omega_u$ and (\ref{santana}) of Lemma \ref{kiss}, then $L(u)\subset \Pi_y$. Given that $x\in L(u)\cap \bd G$, it follows that $x\in G_y$. Furthermore, since the relation $G_y=S(K,y)\cap S(K,-y)$ holds, we deduce that $L(x,y)$ is supporting line of $K$. Therefore the cone $S(K,x)$ can be represented as
\[
S(K,x)=\bigcup_{y\in S\partial(S,u)} L(x,y).
\]
Since $S$ is an ellipsoid, the set $S\partial(S,u)$ is an ellipse. Thus $S(K,u)$ is an ellipsoidal cone. Thus $K$ is an ellipsoid.

\section{Proof of Theorem \ref{valor} for dimension 3.} 
By Theorem \ref{chiquita} is enough to prove that $S$ is an ellipsoid. In order to prove that $S$ is an ellipsoid we will apply Kakutani's Theorem \cite{Kakutani}: \textit{if for every hyperplane $\Lambda$, passing through a fix point $O\in \inte K$, there exist a line $L_{\Lambda}$ such that 
\[
\Lambda \cap K\subset S\partial(K,L_{\Lambda}),
\]
then $K$ is an ellipsoid.}
\begin{figure}[H]
    \centering
    \includegraphics[width=.9\textwidth]{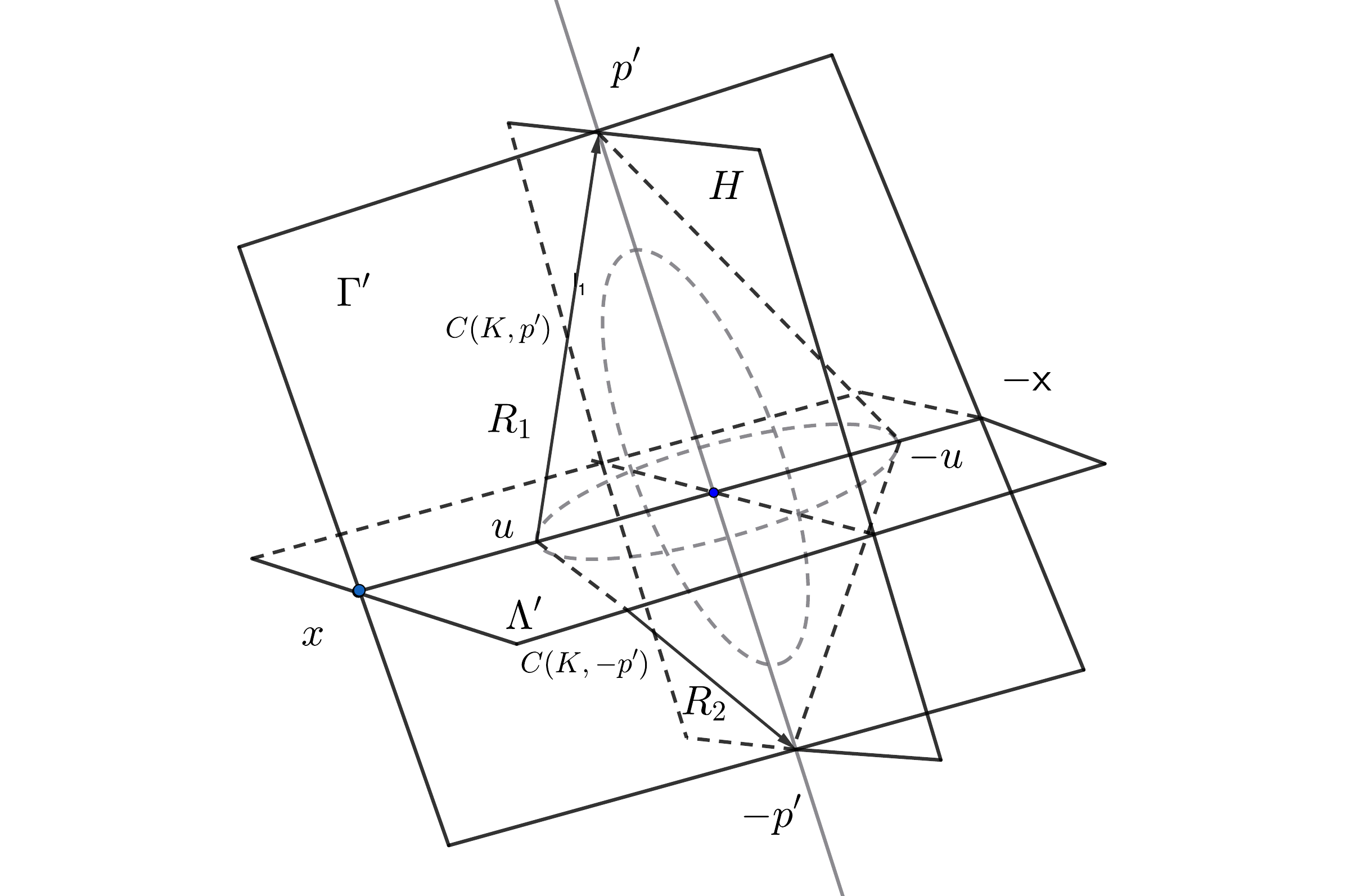}
    \caption{Given the plane $\Lambda$, $O\in \Lambda$, there exists $p\in \bd S$ such that \\$C(K,p)\cap C(K,-p)=\Lambda \cap G$.
}
    \label{aguadona}
\end{figure}
Let $\Lambda$ be a plane, $O\in \Lambda$. Let $x\in \Lambda \cap \bd S$   and let $\{u,-u\}:= L(x,-x)\cap \bd G$, notice that, by Theorem \ref{great}, $S$ and $G$ are centrally symmetric. Let $\Gamma'$ be a plane containing $L(x,-x)$. We denote by $R_1, R_2\subset \Gamma'$ the rays emanating from $u$ which are contained in the supporting lines $L_1$, $L_2$ of $\Gamma' \cap K$ passing through $u$ and let $p':=R_1\cap \bd S$ and $q':=R_2\cap \bd S$ (notice that here we use the condition $G\subset \inte S$). By virtue of the hypothesis it follows that $O\in L(p',q')$, i.e. $q'=-p'$. Furthermore there exists a plane $\Lambda'$ such that $L(u,-u)\subset \Lambda'$ and 
\begin{eqnarray}\label{mamonsita}
C(K,p')\cap C(K,-p')=\Lambda'\cap G.
\end{eqnarray}
Since $L(u,-u)\subset \Lambda'$ and $L(u,-u)=L(x,-x)$ it follows that $L(x,-x)\subset \Lambda'$. Thus, by Lemma \ref{muneca}, $L(p',-p')\subset H$ (see Fig. \ref{aguadona}), where $H$ is a plane such that $O\in H$ and  
\begin{eqnarray}\label{sensual}
C(K,x)\cap C(K,-x)=H\cap G,
\end{eqnarray}
Varying $\Gamma'$, always keeping  the condition $L(x,-x) \subset \Gamma'$, we can find a position of $\Gamma'$, $p'$, which will be denote by $\Gamma$, $p$, respectively, such that the condition \ref{mamonsita} holds, i.e.
\begin{eqnarray}\label{antojo}
C(K,p)\cap C(K,-p)=\Lambda\cap G.
\end{eqnarray}
On the other hand, by Lemma \ref{porton}, the support plane $\Gamma_x$ is parallel to $L(p,-p)$. Hence 
\[
\Lambda \cap S\subset S\partial (S,L(p,-p)).
\]
Hence, by Kakutani's Theorem $K$ is an ellipsoid.
  
\section{Reduction of the general case of Theorems \ref{valor}, \ref{great} and \ref{chiquita} to dimension 3.} 
Suppose that $n\geq 4$ and that the convex body $K\subset \Rn$ has the strong intersection property in dimension $n$ relative to the point $O\in \inte K$, the body $S$, $K\subset \inte S$ and with associated body $G$, $K\subset \inte G$. Let $\Gamma$ be a hyperplane, $O\in \Gamma$. We claim that $K \cap \Gamma$ has the strong intersection property, in dimension $n-1$, relative to $O$ and $S\cap \Gamma$ and with associated body $G\cap \Gamma$. 
 
Let $x\in S \cap \Gamma$. By hypothesis there exists $y\in \bd S$ and a hyperplane $\Pi$, $O\in \Pi$, such that 
\[
S(K,x)\cap S(K,y)= \Pi \cap G.
\]
Notice that, since $L(x,O)\subset \Gamma$, $y\in \Gamma$. Hence $y\in S \cap \Gamma$. It follows that 
\[
S(K\cap \Gamma,x)  \cap S(K\cap \Gamma,y)=(S(K,x) \cap \Gamma)\cap (S(K,y)\cap \Gamma)= (\Pi \cap G) \cap \Gamma,
\]
i.e., $K \cap \Gamma $ has the strong intersection property in dimension $n-1$ relative to $O$ and $S\cap \Gamma$ and with associated body $G\cap \Gamma$. 

\textbf{Reduction of the general case of Theorem \ref{great} to dimension 3.} If we assume that the convex body $K\subset \Rn$, $n\geq 3$ has the strong intersection property in dimension $n$ 
relative to the point $O\in \inte K$, the body $S$, $K\subset \inte S$, and with associated strictly convex body $G$, $K\subset \inte G$, and that Theorem \ref{great} holds in dimension $n-1$, by virtue of the observation at the beginning of this section, it follows that all the sections of convex body $K$ with hyperplanes passing through $O$ are $O$-symmetric. Then $K$ is $O$-symmetric.  

Since it has been proved the case $n=3$ of the Theorem \ref{great}, the proof of the Theorem \ref{great} now is complete.

\textbf{Reduction of the general case of Theorems \ref{valor} and \ref{chiquita} to dimension 3.} If we suppose that the convex body $K\subset \Rn$, $n\geq 3$ has the strong intersection property in dimension $n$ relative to the point $O\in \inte K$, the body $S$, $K\subset \inte S$, and with associated strictly convex body $G$, $K\subset \inte G$, and that Theorems \ref{valor} and \ref{chiquita} holds in dimension $n-1$, by virtue of the observation at the beginning of this section, it follows that all the sections of convex bodies $K, S$ and $G$ with hyperplanes passing through $O$ are homothetic $(n-1)$-ellipsoids. Then, by Theorem 16.12 of \cite{bu} and a theorem of \cite{Rogers} (see \textbf{a) 2.}), $K,S$ and $G$ are $O$-symmetric homothetic $n$-ellipsoids.  

Since it has been proved the case $n=3$ of the Theorems \ref{valor} and \ref{chiquita}, the proof of the Theorem \ref{great} now is complete.

\end{document}